\documentclass[lualatex,12pt]{article}

\usepackage{amsmath}
\usepackage{amssymb}
\usepackage{amsthm}
\usepackage{tikz,pgf}
\usepackage{float}
\usepackage{tikz-3dplot}

\newtheorem{thm}{Theorem}
\newtheorem{prop}{Proposition}

\newtheorem{cor}{Corollary}
\newtheorem{exm}{Example}
\newtheorem{rmk}{Remark}
\newtheorem{definition}{Definition}

\begin{document}

\begin{center}
 {\large An ${\mathfrak S}_3$-cover of $K_4$ and integral polyhedral graphs}

 \bigskip
 Taizo Sadahiro
\end{center}

\begin{abstract}
We first show that the star graph
$
 X_n=\mathrm{Cay}(\mathfrak S_{n+1},S),
$
where $S=\{(i, n+1)\,|\, 1\leq i\leq n\}$,
is an ${\mathfrak S}_n$-Galois cover of the complete graph $K_{n+1}$.
We then study the case $n=3$ in detail, 
where the intermediate covers are the cube graph and the truncated tetrahedron graph.
Using the factorization of the Ihara zeta function along coverings,
we derive a multiset identity relating the adjacency spectra of these graphs,
yielding a unified explanation for the integrality of their spectra.
This construction yields an explicit integral ${\mathfrak S}_3$-cover of $K_5$.
\end{abstract}




%
%
%
%

\section{Introduction}

It is well known that the tetrahedron $K_4$,
the cube $Q$, and the truncated tetrahedron $T$ have
only integral eigenvalues when viewed as graphs.
The following fact is less well known:
The adjacency operator of the honeycomb lattice 
under the periodic boundary
condition shown in Figure $\ref{fig:boundarycondition}$
has the integral eigenvalue multiset expressed by the Fourier formula,
\[
 {\mathcal S}=\left\{\pm\left|1+\exp\left({\frac{2\pi{k}{i}}{6}}\right) +
 \exp\left({\frac{2\pi{(-k+3l)}{i}}{6}}\right) \right|~~\middle|\,
 k\in {\mathbb Z}/6{\mathbb Z},~ l\in {\mathbb Z}/2{\mathbb Z}\right\}.
\]
We give a unifying interpretation of these facts
using the graph Galois theory.
In particular, the Galois theory shows the following multisets relation,
\[
 {\mathcal S} \cup {\rm Spec}(K_4) \cup {\rm Spec}(K_4)
 =
 {\rm Spec}(Q) \cup {\rm Spec}(T) \cup {\rm Spec}(T),
\]
where ${\rm Spec}(\Gamma)$ is the multiset of the eigenvalues
of the adjacency matrix of a finite graph $\Gamma$.
The integrality of ${\mathcal S}$ is derived from the fact that
it is identical to the spectrum of a {\em star graph}
whose integrality was conjectured by Abdollahi and Vatandoost
\cite{abdollahi2009cayley} and
proved by Krakovski and Mohar \cite{krakovski2012spectrum}.

Our main aim in this paper is to  present a classical but underappreciated example of 
a non-abelian Galois cover of $K_4$ with the Galois group ${\mathfrak S}_3$, 
whose intermediate covers are the cube and the truncated tetrahedron. 
Although the graph is small and explicit, 
this cover seems to be not known in the literature. 
We give a Cayley graph realization and compute the relation
between the Ihara zeta functions of the intermediate covers, 
clarifying its connection to examples mentioned by Stark and Terras in \cite{stark2000zeta}. 
We further show some geometric properties,
how this cover embeds in the torus or the honeycomb lattice,
which enbales the Fourier analysis.
As a biproduct, we also show an explicit integral ${\mathfrak S}_3$-cover of $K_5$.

\begin{center}
 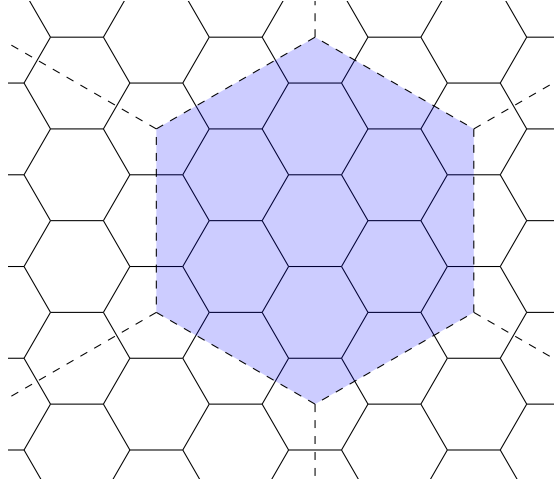
\begin{figure}[H]
  \begin{center}
   \begin{tikzpicture}[scale=0.7]
    \clip (-5.3,-4) -- (5,-4) -- (5,5) -- (-5.3,5) -- cycle;
    \draw[dashed,fill=blue, fill opacity=0.2] (-2.5,{-sqrt(3)/2}) -- ++(-3,{-sqrt(3)});
    \draw[dashed,fill=blue, fill opacity=0.2] (-2.5,{-sqrt(3)/2})  ++(3,{-sqrt(3)}) -- ++(0,{-2*sqrt(3)});
    \draw[dashed,fill=blue, fill opacity=0.2] (-2.5,{-sqrt(3)/2})  ++(3,{-sqrt(3)})++(0,{4*sqrt(3)}) -- ++(0,{2*sqrt(3)});
    \draw[dashed,fill=blue, fill opacity=0.2] (-2.5,{-sqrt(3)/2})  ++(3,{-sqrt(3)}) ++(3,{sqrt(3)}) -- ++(3,{-sqrt(3)});
    \draw[dashed,fill=blue, fill opacity=0.2] (-2.5,{-sqrt(3)/2})  ++(6,{2*sqrt(3)}) -- ++(3,{sqrt(3)});
    \draw[dashed,fill=blue, fill opacity=0.2] (-2.5,{-sqrt(3)/2})  ++(0,{2*sqrt(3)}) -- ++(-3,{sqrt(3)});
    \foreach \y in {-4,-3,...,4}{
    \foreach \x in {-4,-3,...,4}{
    \draw ({(\x+\y)*3/2},{(\x-\y)*sqrt(3)/2}) -- ++(1,0);
    \draw ({(\x+\y)*3/2},{(\x-\y)*sqrt(3)/2}) -- ++({-1/2},{sqrt(3)/2});
    \draw ({(\x+\y)*3/2},{(\x-\y)*sqrt(3)/2}) -- ++({-1/2},{-sqrt(3)/2});
    }}
    \draw[dashed,fill=blue, fill opacity=0.2] (-2.5,{-sqrt(3)/2}) -- ++(3,{-sqrt(3)})
    -- ++(3,{sqrt(3)}) -- ++(0,{2*sqrt(3)}) -- ++(-3,{sqrt(3)}) -- ++(-3,{-sqrt(3)}) -- cycle;
   \end{tikzpicture}
  \end{center}
  \caption{A periodic boundary condition with hexagonal fundamental domain
  which make the adjacency operator on the honeycomb lattice to have only integral eigenvalues.}
  \label{fig:boundarycondition}
 \end{figure}
\end{center}

\section{Construction of ${\mathfrak S}_n$-cover of $K_{n+1}$}

\begin{definition}[Graph Cover]
Let \(X = (V,E)\) and \(\tilde{X} = (\tilde{V}, \tilde{E})\) be finite simple graphs.  
A graph \(\tilde{X}\) is called a \emph{cover} of \(X\) if there exists a surjective graph morphism
\[
p : \tilde{X} \to X
\]
such that for every vertex \(\tilde{v} \in \tilde{V}\), the map \(p\) induces a bijection between the set of edges incident to \(\tilde{v}\) and the set of edges incident to \(p(\tilde{v})\).  
In particular, the local neighborhood structure is preserved.
\end{definition}

\begin{definition}[Galois Cover]
A cover \(p : \tilde{X} \to X\) is called \emph{Galois} (or \emph{regular}) if there exists a group \(G\) acting freely and transitively on each fiber \(p^{-1}(w)\) for every vertex \(w \in V(X)\),  
and such that the quotient graph \(\tilde{X} / G\) is isomorphic to \(X\).  

The group \(G\) is called the \emph{Galois group} of the cover,
and $\tilde{X}$ is called a \emph{$G$-cover} of $X$.
\end{definition}

The ${\mathfrak S}_n$-cover of $K_{n+1}$ we construct here is known as
the star graph, which is the Cayley graph of ${\mathfrak S}_{n+1}$
generated by the {\em star transpositions}.
The star graph is studied in 
\cite{ehrlich1973loopless} in the context of an algorithm generating
permutations.
Its spectral properties has been closely studied.
\begin{thm}
 \label{thm:main}
 Let $\tau_{i}=(i, n+1)\in {\mathfrak S}_{n+1}$ be the transposition
 for $i=1,2,\ldots,n$, and
 let $X_n={\rm Cay}({\mathfrak S}_{n+1}, \{\tau_1, \tau_2, \ldots, \tau_n\})$.
 Then, $X_n$ is an ${\mathfrak S}_n$-cover of the complete graph $K_{n+1}$.
\end{thm}
\begin{proof}
Let the vertex set of $K_{n+1}$ be $V(K_{n+1})=\{1,2,\ldots,n+1\}$.
Then there exists exactly one edge from $i\in V(K_{n+1})$ to $j\in V(K_{n+1})$
if and only if $i\neq j$.
First we define the map $p_V: V(X_n)={\mathfrak S}_{n+1}\to V(K_{n+1})$ by
\[
 p_V(\xi) = \xi(n+1),
\]
which is clearly surjective.
$X_n$ has an edge $e_i$ from $\xi$ to $\eta$ if and only if there exists
a $\tau_i$ such that
\[
 \eta = \xi\tau_i.
\]
Then, $p_V(\eta)=\xi\tau_i(n+1)=\xi(i)$.
Thus $p_V$ induces a bijection from $E_\xi(X_n)$ to $E_{p_V(\sigma)}(K_{n+1})$.
These local bijections form a map $p_E: E(X_n)\to E(K_{n+1})$ and
$p = (p_V, p_E)$ is a covering map.

Let $G_n$ be a subgroup of ${\mathfrak S}_{n+1}$ defined by
\[
 G_n=\left\{\sigma\in {\mathfrak S}_{n+1}\,|\, \sigma(n+1)=n+1\right\}.
\]
Then $G_n$ is isomorphic to ${\mathfrak S}_n$ and
acts naturally on $V(X_n)$ from the {\em right}. 
First we show that
this action of $\sigma\in G_n$ on $V(X_n)$ induces
a graph automorphism of $X_n$.
Let $\xi$ and $\eta$ be two adjacent vertices in $X_n$.
That is, there is a transposition $\tau_i$, such that
\[
 \xi\tau_i = \eta.
\]
Then for every $\sigma\in G_n$, we have
\[
 \xi\sigma\tau_{\sigma^{-1}(i)} = \xi\tau_{i}\sigma = \eta\sigma,
\]
which implies $\xi\sigma$ is adjacent to $\eta\sigma$.
Thus $\sigma$'s action on $V(X_n)$ induces
a graph automorphism of $X_n$.

It is clear that
\[
 p_V(\xi\sigma) = p_V(\xi),
\]
for all $\xi\in {\mathfrak S}_{n+1}$ and $\sigma\in G$,
since $\xi\sigma(n+1)=\xi(n+1)$. Thus $G$ acts on 
each fiber $p^{-1}(i)$ freely and transitively for $i=1,2,\ldots,n+1$.

\end{proof}

\begin{rmk}
 By Theorem ${\ref{thm:main}}$, we can construct
 ${\mathfrak S}_n$-cover of any connected simple graphs
 with $n+1$ vertices, by deleting edges from $K_{n+1}$
 and corresponding edges in the cover $X_n$.
 This includes the example shown by Stark and Terras \cite{stark2000zeta} of
 ${\mathfrak S}_3$-cover of a graph $K_4-\{e\}$,
 a graph obtained from $K_4$ by deleting an edge $e$.
\end{rmk}

The integrarity of the star graph was first proved by
Krakovski and Mohar \cite{krakovski2012spectrum}.
\begin{thm}
[Krakovski and Mohar \cite{krakovski2012spectrum}]
\label{thm:kramoha}
When $n\geq 3$, $k\in {\rm Spec}(X_n)$ for $k=0,1,2,\ldots,n$.
\end{thm}
Moreover, Chapuy and F\'{e}ray \cite{ChapuyFeray} gave an explicit formula for the multiplicity of each eigenvalue.  
\begin{thm}
[Chapuy and F\'{e}ray \cite{ChapuyFeray}]
\label{thm:chapuyferay}
Let ${\mathcal P}(n+1)$ denote the set of partitions of $n+1$, 
let ${\rm SYT}(\lambda)$ be the set of standard Young tableaux  of 
shape $\lambda\in{\mathcal P}(n+1)$,
and let $f^\lambda=|SYT(\lambda)|$.  
For a tableau $T \in \mathrm{SYT}(\lambda)$, 
let $c_T(m)$ be the \emph{content} of the box containing $m$, 
i.e., the column index minus the row index of that box.  
Then the multiplicity of the eigenvalue $k$ of $X_n$ is given by
\[
\mathrm{mult}_{X_n}(k) = \sum_{\lambda \in {\mathcal P}(n+1)} f^\lambda \, I_\lambda(k),
\]
where
\[
I_\lambda(k) = \#\{\, T \in \mathrm{SYT}(\lambda) \,|\, c_T(n+1) = k \,\}.
\]
\end{thm}
This formula follows from the decomposition of the permutation representation of 
${\mathfrak S}_{n+1}$ and the Jucys–Murphy element diagonalization.

\bigskip
In the rest of this section, we consider the case of $n=3$,
 when Theorem $\ref{thm:main}$ yields
 an ${\mathfrak S}_3$-cover of $K_4$ depicted in the left side of
 Figure $\ref{fig:truncatedtetrahedron}$.
 A figure of this graph appears in The art of Computer Programming vol.$4$ 
 \cite{knuth2011taocp4a} Fig.$44$.
Before stating our main results, we introduce some additional terminology on graph covers.
\begin{definition}[Intermediate cover]
Let $\widetilde{X}$ be a cover of a graph $X$ whose covering map is
\(p:\widetilde{X}\to X\).
An \emph{intermediate cover} is a graph $Y$
together with a covering map \(q:\widetilde{X}\to Y\) and a covering map \(r:Y\to X\) such that
\[
p = r \circ q.
\]
In other words, \(Y\) is a graph through which \(p\) factors as a composition of two covers.
\end{definition}

For the proof of the following theorem, see \cite{terras2010zeta}.
\begin{thm}[Galois correspondence for graph covers \cite{terras2010zeta}]
Let \(p:\tilde{X}\to X\) be a Galois cover with Galois group \(G\).
There is a one-to-one inclusion-reversing correspondence between
\[
\{\mbox{subgroups } H\le G\}
\quad\mbox{and}\quad
\{\mbox{intermediate covers } \tilde{X}\to Y\to X\},
\]
given by \(H\mapsto Y=\tilde{X}/H\).
Moreover, for any \(H\le G\):
\begin{enumerate}
  \item The quotient map \(\tilde{X}\to\tilde{X}/H\) is a covering map.
  \item \(\tilde{X}/H\to X\) is a Galois  cover if and only if \(H\triangleleft G\);
  in that case its Galois group is \(G/H\).
\end{enumerate}
\end{thm}

\begin{thm}
 \label{thm:GaloisX3}
 The star graph $X_3$ is a ${\mathfrak S}_3$-cover of
 $K_4$, whose intermediate covers are isomorphic to
 the cube graph or the truncated tetrahedron graph.
\end{thm}
\begin{proof}
 Since the star graph $X_3$ is a ${\mathfrak S}_3$-cover of $K_4$,
 we have intermediate covers corresponding to the non-trivial
 four subgroups $\langle(1,2)\rangle, \langle(2,3)\rangle, \langle(1,3)\rangle,$
 and $\langle(1,2,3)\rangle$.

 By identifying the $X_3$'s vertices in the orbit of the action by
 subgroup $\langle(1,2)\rangle$, we obtain the quotient graph
 $X_3/\langle(1,2)\rangle$ which is isomorphic the truncated
 tetrahedron graph $T$. See Figure $\ref{fig:truncatedtetrahedron}$.
 By symmetry, we see that $X_3/\langle(1,2)\rangle \cong X_3/\langle(2,3)\rangle \cong X_3/\langle(1,3)\rangle$.

\begin{center}
 \begin{figure}[H]
  \begin{center}
   \begin{tikzpicture}
    \tiny
    \begin{scope}
    \clip (1,0)  +(0:1) -- +(30:3.5) -- +(90:3.5) -- +(150:3.5) -- +(210:3.5) -- +(270:3.5) -- +(330:3.5) -- +(30:3.5);
    \begin{scope}[xshift=1cm]
    \coordinate (v0) at (0:1);
    \coordinate (v1) at (60:1);
    \coordinate (v2) at (120:1);
    \coordinate (v3) at (180:1);
    \coordinate (v4) at (240:1);
    \coordinate (v5) at (300:1);
    \draw[fill = blue, opacity=0.3] (v0) -- (v1) -- (v2) -- (v3) -- (v4) -- (v5) -- cycle;
    \end{scope}

    \begin{scope}[xshift=4cm]
    \coordinate (v0) at (0:1);
    \coordinate (v1) at (60:1);
    \coordinate (v2) at (120:1);
    \coordinate (v3) at (180:1);
    \coordinate (v4) at (240:1);
    \coordinate (v5) at (300:1);
    \draw[fill = orange, opacity=0.3] (v0) -- (v1) -- (v2) -- (v3) -- (v4) -- (v5) -- cycle;
    \end{scope}

    \begin{scope}[xshift=-2cm]
    \coordinate (v0) at (0:1);
    \coordinate (v1) at (60:1);
    \coordinate (v2) at (120:1);
    \coordinate (v3) at (180:1);
    \coordinate (v4) at (240:1);
    \coordinate (v5) at (300:1);
    \draw[fill = orange, opacity=0.3] (v0) -- (v1) -- (v2) -- (v3) -- (v4) -- (v5) -- cycle;
    \end{scope}

    \begin{scope}[xshift=2.5cm, yshift=-{1.5*1.732}cm]
    \coordinate (v0) at (0:1);
    \coordinate (v1) at (60:1);
    \coordinate (v2) at (120:1);
    \coordinate (v3) at (180:1);
    \coordinate (v4) at (240:1);
    \coordinate (v5) at (300:1);
    \draw[fill = red, opacity=0.3] (v0) -- (v1) -- (v2) -- (v3) -- (v4) -- (v5) -- cycle;
    \end{scope}
    \begin{scope}[xshift=-{0.5}cm, yshift=-{-1.5*1.732}cm]
    \coordinate (v0) at (0:1);
    \coordinate (v1) at (60:1);
    \coordinate (v2) at (120:1);
    \coordinate (v3) at (180:1);
    \coordinate (v4) at (240:1);
    \coordinate (v5) at (300:1);
    \draw[fill = red, opacity=0.3] (v0) -- (v1) -- (v2) -- (v3) -- (v4) -- (v5) -- cycle;
    \end{scope}

    \begin{scope}[xshift=2.5cm, yshift=-{-1.5*1.732}cm]
    \coordinate (v0) at (0:1);
    \coordinate (v1) at (60:1);
    \coordinate (v2) at (120:1);
    \coordinate (v3) at (180:1);
    \coordinate (v4) at (240:1);
    \coordinate (v5) at (300:1);
    \draw[fill = green, opacity=0.3] (v0) -- (v1) -- (v2) -- (v3) -- (v4) -- (v5) -- cycle;
    \end{scope}
    \begin{scope}[xshift=-{0.5}cm, yshift=-{1.5*1.732}cm]
    \coordinate (v0) at (0:1);
    \coordinate (v1) at (60:1);
    \coordinate (v2) at (120:1);
    \coordinate (v3) at (180:1);
    \coordinate (v4) at (240:1);
    \coordinate (v5) at (300:1);
    \draw[fill = green, opacity=0.3] (v0) -- (v1) -- (v2) -- (v3) -- (v4) -- (v5) -- cycle;
    \end{scope}
    \end{scope}
    
    \input{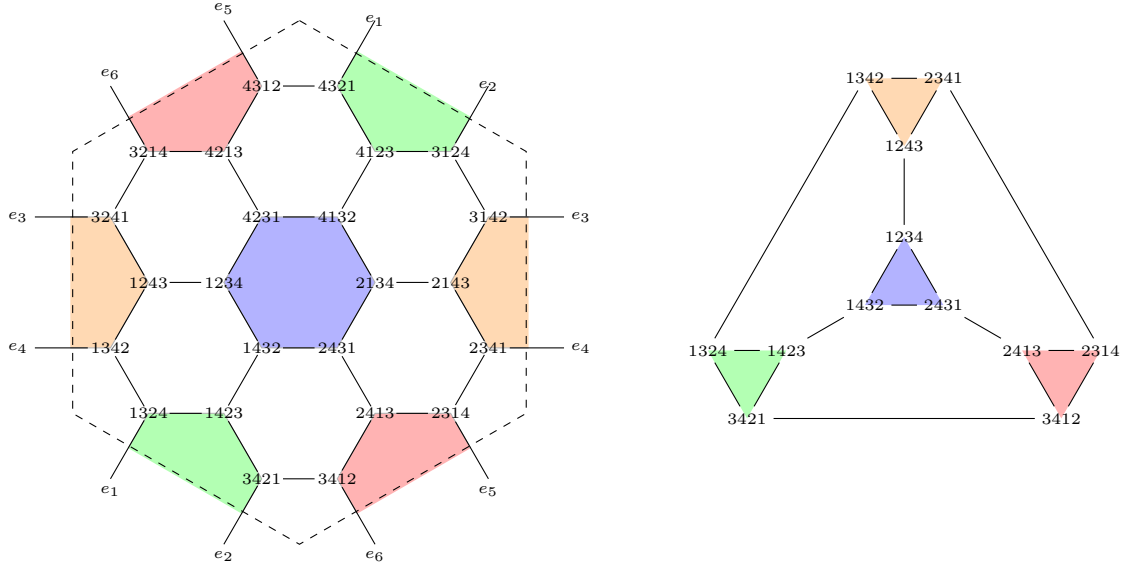}

    \begin{scope}[xshift=9cm,scale=0.6]

     \fill[blue,opacity=0.3](90:1) -- (210:1) -- (330:1) -- cycle;
     \begin{scope}[yshift=4cm]
     \fill[orange,opacity=0.3](30:1) -- (150:1) -- (270:1) -- cycle;
     \end{scope}
     \begin{scope}[xshift=-{2*1.732}cm, yshift=-2cm]
     \fill[green,opacity=0.3](30:1) -- (150:1) -- (270:1) -- cycle;
     \end{scope}
     \begin{scope}[xshift=-{-2*1.732}cm, yshift=-2cm]
     \fill[red,opacity=0.3](30:1) -- (150:1) -- (270:1) -- cycle;
     \end{scope}

    \draw[] (90:1) node (1a) {$1234$};
    \draw (210:1) node (2a) {$1432$};
    \draw (330:1) node (3a) {$2431$};

    \draw (90:4)+(-90:1) node (4b) {$1243$};
    \draw (90:4)+(30:1) node (3b) {$2341$};
    \draw (90:4)+(150:1) node (2b) {$1342$};

    \draw (210:4)+(30:1) node (4c) {$1423$};
    \draw (210:4)+(150:1) node (1c) {$1324$};
    \draw (210:4)+(270:1) node (3c) {$3421$};

    \draw[] (-30:4)+(150:1) node (4d) {$2413$};
    \draw (-30:4)+(270:1) node (2d) {$3412$};
    \draw (-30:4)+(30:1) node (1d) {$2314$};

    \draw (1a) -- (2a) -- (3a) -- (1a);
    \draw (3b) -- (2b) -- (4b) -- (3b);
    \draw (3c) -- (4c) -- (1c) -- (3c);
    \draw (1d) -- (4d) -- (2d) -- (1d);

    \draw (1a) -- (4b);
    \draw (2a) -- (4c);
    \draw (3a) -- (4d);

    \draw (2b) -- (1c);
    \draw (3c) -- (2d);
    \draw (1d) -- (3b);
    \end{scope}
   \end{tikzpicture}
   \caption{Left: $X_3={\rm Cay}({\mathfrak S}_4,\{(1,4),(2,4),(3,4)\})$.
   Right: The orbits of the right action by the subgroup $\langle(1,2)\rangle\in G_3\cong {\mathfrak S}_3$
   makes a quotient graph of $X_3$ which is isomorphic to the truncated tetrahedron
   graph.}
   \label{fig:truncatedtetrahedron}
  \end{center}
 \end{figure}
\end{center}

\begin{center}
 \begin{figure}[H]
  \begin{center}
   \begin{tikzpicture}
     \draw (45:1) node[draw, inner sep=1] (1i) {$1$};
     \draw (135:1) node[draw, inner sep=1,circle] (2i) {$2$};
     \draw (225:1) node[draw, inner sep=1] (3i) {$3$};
     \draw (315:1) node[draw, inner sep=1,circle] (4i) {$4$};

     \draw (45:3) node[draw, inner sep=1,circle] (3o) {$3$};
     \draw (135:3) node[draw, inner sep=1] (4o) {$4$};
     \draw (225:3) node[draw, inner sep=1,circle] (1o) {$1$};
     \draw (315:3) node[draw, inner sep=1] (2o) {$2$};

     \draw (1i) -- (2i) -- (3i) -- (4i) -- (1i);
     \draw (1o) -- (2o) -- (3o) -- (4o) -- (1o);

     \draw (1i) -- (3o);
     \draw (2i) -- (4o);
     \draw (3i) -- (1o);
     \draw (4i) -- (2o);
   \end{tikzpicture}
   \caption{The orbits of the right action by the subgroup $\langle(1,2,3)\rangle$
   make a quotient graph of $X_3$ which is isomorphic to the cube graph.
   The vertex marked \fbox{$i$} represents the orbit containing
   the even permutation $\xi\in {\mathfrak S}_{4}$ with $\xi(4)=i$,
   and that marked \textcircled{$i$} corresponds to the odd permutations.
   }
   \label{fig:cube}
  \end{center}
 \end{figure}
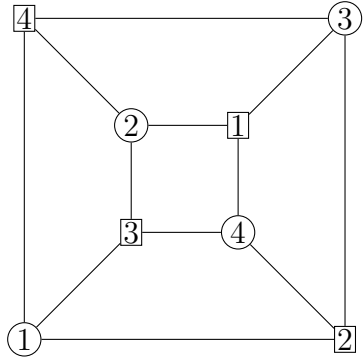 
\end{center}
 In the same manner we can confirm that the quotient graph
 $X_3/\langle(1,2,3)\rangle$ is isomorphic to the cube graph $Q$.
 See Figure $\ref{fig:cube}$.
\end{proof}

\section{Spectral properties}

In the same manner as in \cite{ishikawa2024non},
we use the Ihara zeta function
to study the spectral properties of the graphs,
where we  view an undirected graph
as a bidirected graph, that is,
each edge is replaced by two directed edges of
opposite directions.
The opposite of a directed edge $e$ is denoted $\overline{e}$.

Let $C=(e_0,e_1,\ldots, e_{l-1})$ be a cycle in
a graph $\Gamma$.
Then $C$ is {\em non-backtracking} if
\[
 e_i\neq \overline{e}_{i+1} 
\]
for every $i\in {\mathbb Z}/l{\mathbb Z}$.
The {\em length} $\nu(C)$ of a cycle $C$ is defined by
\[
 \nu(C) = l.
\]
A non-backtracking cycle $C$ is a {\em prime cycle} 
if there exists no pair of a cycle $D$ and an integer $k>1$ such that
\[
 C = D^k.
\]
Let $C=(e_0,e_1,\ldots, e_{l-1})$ be a cycle in a
graph $\Gamma$.
We  introduce an equivalence relation $\sim$ to the set of prime cycles
in $\Gamma$ as follows.
Let $C$ and $C'$ be two prime cycles in $\Gamma$.
Then we define that $C\sim C'$ if  there exists a $k\in {\mathbb Z}/l{\mathbb Z}$ such that
\[
 C' = (e_k, e_{k+1}, \ldots, e_{k+l-1}).
\]
That is, $C'$ is obtained from $C$ by changing the starting vertex.
This relation defines a equivalence relation on the set of all prime cycles in $\Gamma$,
and we denote the equivalence class $[C]$
to which $C$ belongs.
We call $[C]$ a {\em prime} in $\Gamma$.

The {\em zeta function} $\zeta_{\Gamma}$ of a graph 
$\Gamma$ is defined by
\[
 \zeta_{\Gamma}(u) = 
 \prod_{[C]}
 \left(1-u^{\nu(C)}\right)^{-1},
\]
where $[C]$ runs through the primes in $\Gamma$.
The following theorem shows how the zeta functions
are related to the spectra of $\Gamma$.
\begin{thm}
 \label{thm:zeta}
 {\rm (\cite{ihara1966discrete, bass1992ihara})}
 ~Let $A$ be the adjacency matrix of $\Gamma$, and
 let $Q$ be the diagonal matrix
 whose diagonal entry corresponding to the vertex $v$ is ${\rm deg}(v)-1$.
 Then we have
 \[
  \zeta_{\Gamma}(u)
 =
 \left(
 (1-u^2)^{r-1}
 \det\left(I - uA + u^2Q\right)\right)^{-1},
 \]
 where $r-1=\frac{1}{2}{\rm Tr}\left(Q-I\right)$.
\end{thm}
When we apply Theorem $\ref{thm:zeta}$ to $X_n$ and its quotients
by $H\leq {\mathfrak S}_n$,
we have $Q=(n-1)I$ and
\begin{eqnarray}
 \zeta_{X_n/H}(u) & = & 
  (1-u^2)^{-\frac{(n+2)(n+1)!}{2|H|}}u^{-\frac{(n+1)!}{|H|}}
 \det\left(\frac{(n-1)u^2+1}{u}I-A\right)^{-1}\\
 & = & 
  (1-u^2)^{-\frac{(n+2)(n+1)!}{2|H|}}u^{-\frac{(n+1)!}{|H|}}
  P_{X_n/H}\left(\frac{(n-1)u^2+1}{u}\right)^{-1},
 \label{eq:zetacharacteristic}
\end{eqnarray}
where $P_{\Gamma}(x)$ is the characteristic polynomial
of the adjacency matrix of a graph $\Gamma$.

Let $\Gamma/\Delta$ be a Galois cover with the covering map $p: \Gamma\to \Delta$.
Let $C=(e_0, e_1, \ldots, e_{l-1})$ be a prime cycle of $\Delta$,
and $\widetilde{C}=(f_0,f_1,\ldots, f_{l-1})$ be a lift 
of $C$ in $\Gamma$, that is, $\widetilde{C}$ is a path of length $l$
in $\Gamma$ whose projection $\pi(\widetilde{C})$ onto $\Delta$ is $C$.
Let $o(\widetilde{C})$ be the starting vertex of $\widetilde{C}$ and
$t(\widetilde{C})$ the terminal vertex.
Then, since $\Gamma/\Delta$ is a Galois cover, and
both $o(\widetilde{C})$ and $t(\widetilde{C})$ are projected onto the same vertex $v$, 
there exits a unique automorphism $g$ in ${\rm Gal}(\Gamma/\Delta)$ such
that $o(C)g = t(C)$. This unique automorphism $g$ is denoted
\[
\left[\frac{\Gamma/\Delta}{\widetilde{C}}\right]. 
\]
We call the automorphism $\left[\frac{\Gamma/\Delta}{\widetilde{C}}\right]$
the {\em Frobenius automorphism} of $\widetilde{C}$
associated with the Galois cover $\Gamma/\Delta$.
Then, we introduce the $L$-function.
Let $\rho$ be an irreducible representation of ${\rm Gal}(\Gamma/\Delta)$.
The {\em Artin $L$-function} $L(u, \rho, \Gamma/\Delta)$ 
of $\Gamma/\Delta$ associated
with $\rho$ is defined by
\begin{equation}
 \label{eq:Lfundef}
 L(u, \rho, \Gamma/\Delta) = 
 \prod_{[C]}\det\left(I-\rho\left(\left[\frac{\Gamma/\Delta}{\widetilde{C}}\right]\right)u^{\nu(C)}\right)^{-1},
\end{equation}
where $[C]$ in the product runs through the primes of $\Delta$
and it is known that 
$\det\left(I-\rho\left(\left[\frac{\Gamma/\Delta}{\widetilde{C}}\right]\right)u^{\nu(C)}\right)$ 
does not depend on the choice of $\widetilde{C}$.
The following proposition is shown in
\cite[p.154--155]{terras2010zeta}.
\begin{prop}[\cite{terras2010zeta}]
 \label{prop:factorization}
 Let $\Gamma/\Delta$ be a Galois cover,
 and $\widetilde{\Delta}$ be
 an intermediate cover of $\Gamma/\Delta$.
 Then, an irreducible representation  $\widetilde{\rho}$ of ${\rm Gal}(\widetilde{\Delta}/\Delta)$
 can be lifted to the irreducible representation ${\rho}$
 of $G={\rm Gal}(\Gamma/\Delta)$ and we have
 \begin{equation}
  \label{eq:Lfun1}
  L(u,\rho,\Gamma/\Delta) =  L(u, \widetilde{\rho}, \widetilde{\Delta}/\Delta).
 \end{equation}

 Let $H$ be a (not necessarily normal) subgroup of $G$,
 and let $M = \Gamma/H$ be the quotient graph of $\Gamma$
 by the action of $H$.
 Then $M$ is an intermediate cover of $\Gamma/\Delta$.
 Let $\rho$ be an irreducible representation of $H$ and
 let $\rho^\# = {\rm Ind}_H^G\rho$. Then, we have
 \begin{equation}
  \label{eq:Linducedrep}
  L(u, \rho, \Gamma/M) = L(u, \rho^\#, \Gamma/\Delta).  
 \end{equation}

 The zeta function $\zeta_\Gamma$ is factorized into the products of
 $L$-functions:
 \begin{equation}
  \label{eq:zetafactorization}
  \zeta_\Gamma(u)
 =
 \prod_{\chi\in \widehat{{\rm Gal}(\Gamma/\Delta)}}
 L(u, \chi, \Gamma/\Delta)^{d_\chi},
 \end{equation}
where $\widehat{{\rm Gal}(\Gamma/\Delta)}$ is the
set of the irreducible presentations of ${\rm Gal}(\Gamma/\Delta)$
and $d_\chi$ is the degree of the representation $\chi$.
\end{prop}

In the following theorem,
we show a zeta functions relation of the intermediate coverings 
of the ${\mathfrak S}_3$-cover,
which is shown in \cite{stark2000zeta,terras2010zeta}
for a specific example, 
without the proof for general ${\mathfrak S}_3$-coverings.
\begin{thm}
\label{thm:zeta-relation-S3}
Let $p:Y\to X$ be a Galois cover of finite graphs with
$\mathrm{Gal}(Y/X)\cong {\mathfrak S}_3$.  Let
$Q=Y/\langle(1,2,3)\rangle$ and $T=Y/\langle(1,2)\rangle$ be the
intermediate coverings corresponding to the cyclic subgroups 
$C_3=\langle(1,2,3)$ and $C_2=\langle(1,2)\rangle$ respectively.  
Then we have the identity
\[
\zeta_Y(u)\,\zeta_X(u)^2 = \zeta_Q(u)\,\zeta_T(u)^2.
\]
\end{thm}

\begin{proof}
Write $L(u,\chi,Y/H)$ for the $L$-function of a representation $\chi$ of a
subgroup $H$ (for the covering $Y\to Y/H$).  
Since ${\mathfrak S}_3$ has three  irreducible representations,
the trivial representation $1$, the sign representation ${\rm sgn}$,
and the standard representation ${\rm std}$,
the factorization formula  $(\ref{eq:zetafactorization})$ yields
\begin{equation}
 \label{eq:factorZetaY}
\zeta_Y(u)=\zeta_X(u)\,L(u,\mathrm{sgn},Y/X)\,L(u,\mathrm{std},Y/X)^2.
\end{equation}

Since the subgroup $C_3=\langle(1\,2\,3)\rangle$ is normal in ${\mathfrak S}_3$,
$Q=Y/C_3$ is a Galois cover of $X$ whose Galois group
is ${\mathfrak S}_3/\langle(1,2,3)\rangle\cong C_2$, we have a factorization
\[
 \zeta_Q(u) = \zeta_X(u)L(u, {\rm sgn}_Q, Q/X),
\]
where ${\rm sgn}_Q$ is the sign representation of $Gal(Q/X)$.
By $(\ref{eq:Lfun1})$, we have
\[
L(u, {\rm sgn}, Y/X) = L(u, {\rm sgn}_Q, Q/X),
\]
therefore we have
\begin{equation}
 \label{eq:Lsgn}
L(u,\mathrm{sgn},Y/X)=\frac{\zeta_Q(u)}{\zeta_X(u)}.
\end{equation}

For the subgroup $C_2=\langle(1\,2)\rangle$ note that
$\mathrm{Ind}_{C_2}^{{\mathfrak S}_3}\mathrm{sgn}_{C_2}\cong \mathrm{sgn}\oplus\mathrm{std}$
where ${\rm sgn}_{C_2}$ is the sign representation of $C_2$.
Then, by the induction property $(\ref{eq:Linducedrep})$
of $L$-functions, we have
\[
L(u,\mathrm{sgn},Y/T)=L\big(u,\mathrm{Ind}_{C_2}^{{\mathfrak S}_3}\mathrm{sgn},Y/X\big)
= L(u,\mathrm{sgn},Y/X)\,L(u,\mathrm{std},Y/X).
\]
Applying the factorization property $(\ref{eq:zetafactorization})$ to
the cover $Y\to Y/C_2$ whose Galois group is $C_2$, we have
\begin{eqnarray*}
 \zeta_Y(u) & = & \zeta_T(u) L(u,{\rm sgn}_{C_2}, Y/T)\\
 & = & \zeta_T(u) L(u,{\rm Ind}_{C_2}^{{\mathfrak S}_3}{\rm sgn}_{C_2}, Y/X)\\
 & =& \zeta_T(u)L(u,\mathrm{sgn},Y/X)\,L(u,\mathrm{std},Y/X)\\
 & = & \zeta_T(u)\frac{\zeta_Q(u)}{\zeta_X(u)}L(u,\mathrm{std},Y/X).
\end{eqnarray*}
Thus we have obtained
\begin{equation}
 \label{eq:Lstd}
  L(u,\mathrm{std},Y/X)
  = \frac{\zeta_Y(u)\,\zeta_X(u)}{\zeta_T(u)\,\zeta_Q(u)}.
\end{equation}
Substituting $(\ref{eq:Lsgn})$ and $(\ref{eq:Lstd})$ into
$(\ref{eq:factorZetaY})$ yields the claimed identity
\[
\zeta_Y(u)\,\zeta_X(u)^2 = \zeta_Q(u)\,\zeta_T(u)^2.
\]
This completes the proof.
\end{proof}

\begin{cor}
 \label{cor:specrel}
Let $Y, X, T$ and $Q$ be as in Theorem $\ref{thm:zeta-relation-S3}$.
Then we have
\[
 {\rm Spec}(Y) \cup {\rm Spec}(X) \cup {\rm Spec}(X)
 = 
 {\rm Spec}(Q) \cup {\rm Spec}(T) \cup {\rm Spec}(T).
\]
\end{cor}
\begin{proof}
 This can be expressed in terms of the characteristic polynomials
 of the adjacency matrices,
\begin{equation}
 \label{eq:chararel}
 P_{Y}(x)P_{X}(x)^2 = P_{Q}(x) P_{T}(x)^2,
\end{equation}
which follows from Theorem $\ref{thm:zeta-relation-S3}$ and $(\ref{eq:zetacharacteristic})$.
\end{proof}

\begin{rmk}
In fact, we can directly confirm this spectral relation by calculating the exact eigenvalues
\begin{eqnarray*}
 P_{X_3}(x) & = &(x+3)(x+2)^6(x+1)^3x^4(x-1)^3(x-2)^6(x-3).\\
 P_{K_4}(x) & = &(x+1)^3(x-3). \\
 P_{Q}(x) & = & (x+3)(x+1)^3(x-1)^3(x-3).\\
 P_{T}(x) & = &(x+2)^3(x+1)^3x^2(x-2)^3(x-3).
\end{eqnarray*}
The characteristic polynomial $P_{X_3}(x)$ can be computed using Theorem $\ref{thm:chapuyferay}$.
The following table shows the values $I_\lambda(k)$ and $f^\lambda$:
\[
 \begin{array}{c||c|c|c|c|c|c|c||c}
  \lambda\backslash k & 3 & 2 & 1 & 0 & -1 & -2 & -3 & f^\lambda\\
  \hline
   4 & 1 & 0 & 0 & 0 & 0 & 0 & 0 & 1 \\
  \hline
  31 & 0 & 2 & 0 & 0 & 1 & 0 & 0 & 3\\ 
  \hline
   22 & 0 & 0 & 0 & 2 & 0 & 0 & 0 & 2 \\ 
  \hline
   211 & 0 & 0 & 1 & 0 & 0 & 2 & 0 & 3\\ 
  \hline
   1111 & 0 & 0 & 0 & 0 & 0 & 0 & 1 & 1\\ 
 \end{array}
\]
$P_{K_4}(x)$ and $P_Q(x)$ can also be easily computed since $Q$ is a quadratic
bipartite cover of the complete graph $K_4$.
Thus, assuming $(\ref{eq:chararel})$,
we can also compute $P_{T}(x)$ without the help of computers.
\end{rmk}

\begin{exm}
 The quotient graph of the star graph $X_4$ by the Klein 4-subgroup
 $V=\langle(1,2)(3,4), (1,3)(2,4)\rangle$ is an ${\mathfrak S}_3$-cover of $K_5$,
 which has intermediate covers $Q$ and $T$ corresponding to
 the subgroups $A_4$ and $V(1,2)$ respectively.
 Then we have
\begin{eqnarray*}
 P_{X_4/V}(x) & = &(x+4)(x+2)^{10}(x+1)^4(x-1)^4(x-2)^{10}(x-4),\\
 P_{K_5}(x) & = &(x+1)^4(x-4), \\
 P_{Q}(x) & = & (x+4)(x+1)^4(x-1)^4(x-4),\\
 P_{T}(x) & = &(x+2)^5(x+1)^4(x-2)^5(x-4),
\end{eqnarray*}
where we can confirm the relation in Corollary $\ref{cor:specrel}$
\[
 P_{X_4/V}(x)P_{K_5}(x)^2 = P_Q(x) P_T(x)^2.
\]
\end{exm}

\begin{exm}
 Since $Y=X_5$ is an ${\mathfrak S}_5$-cover of $K_6$,
 it is also an ${\mathfrak S}_3$-cover of
 $X=X_5/{\mathfrak S}_3$. 
 Table $\ref{tab:X5}$ shows the
 spectrum of $X$, $Y$ and their intermediate
 coverings $Q=X_5/C_3$ and $T=X_5/C_2$,
 where we can confirm the relation  in Corollary $\ref{cor:specrel}$.
\begin{table}[ht]
\centering
\begin{tabular}{r||r|r|r|r|r|r|r|r|r|r|r}
 & -5 & -4 & -3 & -2 & -1 & 0 & 1 & 2 & 3 & 4 & 5 \\ 
  \hline
 $Y=X_5$ &   1 &  20 & 105 & 120 &  30 & 168 &  30 & 120 & 105 &  20 &   1 \\ 
 $T=X_5/C_2$ &   0 &   5 &  48 &  78 &  15 &  84 &  15 &  42 &  57 &  15 &   1 \\ 
 $Q=X_5/C_3$ &   1 &  10 &  29 &  36 &  10 &  68 &  10 &  36 &  29 &  10 &   1 \\ 
 $X=X_5/{\mathfrak S}_3$ &   0 &   0 &  10 &  36 &   5 &  34 &   5 &   0 &  19 &  10 &   1 \\ 
\end{tabular}
 \caption{The spectrum of the intermediate covers of the covering $Y/X$.}
 \label{tab:X5}
\end{table}
 
\end{exm}

\section{Geometric properties and Fourier analysis}

The covering properies of the star graph $X_3$ in
Theorem $\ref{thm:GaloisX3}$ can be explained in terms of 
the lattices and the crystallographic groups in ${\mathbb R}^2$.
We first give the formal definition
of the hexagonal lattice ${\mathcal L}$.

The hexagonal lattice graph ${\mathcal L}=(V({\mathcal L}), E({\mathcal L}))$
is an infinite bipartite graph whose vertex set is decomposed into
two disjoint subsets $B$ and $W$, defined by
\begin{equation}
 \label{eq:BW}
 B = (1,0) + {\mathbb Z}{\mathbf v}_1 + {\mathbb Z}{\mathbf v}_2,~~~~
 W = (0,0) + {\mathbb Z}{\mathbf v}_1 + {\mathbb Z}{\mathbf v}_2,
\end{equation}
where 
\begin{equation}
 \label{eq:defbase}
{\mathbf v}_1 = \left(\frac{3}{2}, \frac{\sqrt{3}}{2}\right), 
\mbox{~~ and ~~} {\mathbf v}_2 = \left(\frac{3}{2}, -\frac{\sqrt{3}}{2}\right).
\end{equation}
$V({\mathcal L})=B\cup W$ and a vertex in $B$ (resp. $W$) is called {\em black} (resp. {\em white}).
Two vertices $b\in B$ and $w\in W$ are connected by an edge
if and only if 
$\|b-w\|=1$,
and there is no edge connecting two vertices $v$ and $w$ if
$\{v,w\}\subset B$ or $\{v,w\}\subset W$.

 We define two sublattices $\Lambda_Q$ and $\Lambda_{X_3}$ of 
 the lattice ${\mathbb Z}{\mathbf v}_1 + {\mathbb Z}{\mathbf v}_2$
 by
 \[
  \Lambda_{Q} = {\mathbb Z}(2{\mathbf v}_1) + {\mathbb Z}(2{\mathbf v}_2),
  ~~ \mbox{ and } ~~
  \Lambda_{X_3} = {\mathbb Z}(2{\mathbf v}_1 + 2{\mathbf v}_2) + {\mathbb Z}(4{\mathbf v}_1-2{\mathbf v}_2).
 \]
 See Figure $\ref{fig:lattices}$. Then, it is clear that $\Lambda_{X_3}$ is a sublattice of $\Lambda_Q$ and
 \[
 \Lambda_{Q}/\Lambda_{X_3} \cong C_3.
 \]
\begin{center}
 \begin{figure}[H]
  \begin{center}
   \begin{tikzpicture}[scale=0.7]
    \clip (-3.3,-2.7) -- (9.3,-2.7) -- (9.3,7.3) -- (-3.3,7.3) -- cycle;
    \draw[dashed, red, line width = 1] (-2.5,{-sqrt(3)/2}) -- ++(-3,{-sqrt(3)});
    \draw[dashed, red, line width = 1] (-2.5,{-sqrt(3)/2})  ++(3,{-sqrt(3)}) -- ++(0,{-2*sqrt(3)});
    \draw[dashed, red, line width = 1] (-2.5,{-sqrt(3)/2})  ++(3,{-sqrt(3)})++(0,{4*sqrt(3)}) -- ++(0,{2*sqrt(3)});
    \draw[dashed, red, line width = 1] (-2.5,{-sqrt(3)/2})  ++(3,{-sqrt(3)}) ++(3,{sqrt(3)}) -- ++(3,{-sqrt(3)}) -- ++(30:3);
    \draw[dashed, red, line width = 1] (-2.5,{-sqrt(3)/2})  ++(6,{2*sqrt(3)}) -- ++(3,{sqrt(3)})-- ++(90:3);
    \draw[dashed, red, line width = 1] (-2.5,{-sqrt(3)/2})  ++(6,{2*sqrt(3)}) ++(3,{sqrt(3)})-- ++(-30:3);
    \draw[dashed, red, line width = 1] (-2.5,{-sqrt(3)/2})  ++(0,{2*sqrt(3)}) -- ++(-3,{sqrt(3)});
    \draw[dashed, red, line width = 1] (-2.5,{-sqrt(3)/2}) -- ++(3,{-sqrt(3)})
    -- ++(3,{sqrt(3)}) -- ++(0,{2*sqrt(3)}) -- ++(-3,{sqrt(3)}) -- ++(-3,{-sqrt(3)}) -- cycle;
    \foreach \y in {-4,...,7}{
    \foreach \x in {-2,...,7}{
    \draw ({(\x+\y)*3/2},{(\x-\y)*sqrt(3)/2}) -- ++(1,0);
    \draw ({(\x+\y)*3/2},{(\x-\y)*sqrt(3)/2}) -- ++({-1/2},{sqrt(3)/2});
    \draw ({(\x+\y)*3/2},{(\x-\y)*sqrt(3)/2}) -- ++({-1/2},{-sqrt(3)/2});
    }}
    \draw[->,>=latex,line width=2] (60:1) -- +(6,0) node[below]{$2{\mathbf v}_1 + 2{\mathbf v}_2$};
    \draw[->,>=latex,line width=2] (60:1) -- +(60:6) node[above]{$4{\mathbf v}_1 - 2{\mathbf v}_2$};
    \draw[blue, ->,>=latex,line width=2] (60:1) -- +(30:{2*sqrt(3)}) node[above]{$2{\mathbf v}_1$};
    \draw[blue, ->,>=latex,line width=2] (60:1) -- +(-30:{2*sqrt(3)}) node[below]{$2{\mathbf v}_2$};
   \end{tikzpicture}
  \end{center}
  \caption{$\Lambda_{X_3}$ and $\Lambda_Q$}
  \label{fig:lattices}
 \end{figure}
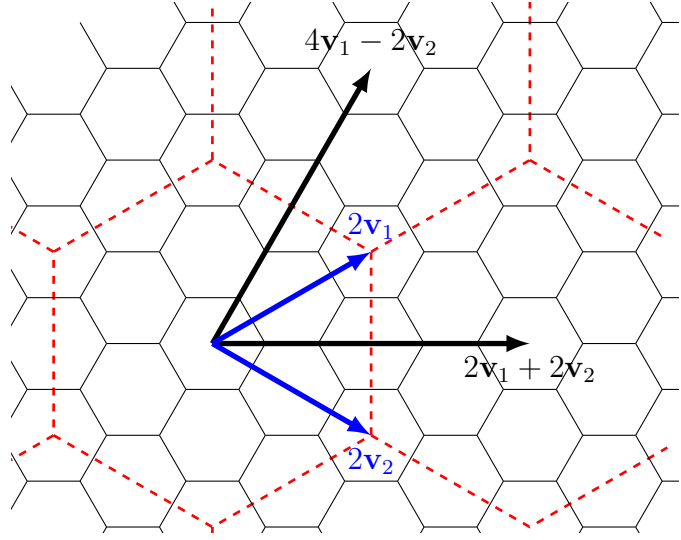
\end{center}

 We identify $\Lambda_Q$ (resp. $\Lambda_{X_3}$) with the group of isometries in ${\mathbb R}^2$
 consisting of the translations by vectors in $\Lambda_Q$ (resp. $\Lambda_{X_3}$),
 and let $G_{K_4}$ (resp. $G_T$) be the group of isometries in ${\mathbb R}^2$ generated by
 the translations by vectors in $\Lambda_{Q}$ (resp. $\Lambda_{X_3}$)
 and the half-turn rotation $R$
 around the origin $(0,0)\in {\mathbb R}^2$.
 Therefore we have
 \[
  G_{K_4} = \Lambda_{Q}\rtimes \langle{R}\rangle,
 \mbox {~~ and ~~}
 G_{T} = \Lambda_{X_3}\rtimes \langle{R}\rangle.
 \]
 Futher, we have
 \[
  G_{K_4}/\Lambda_{X_3} \cong {\mathfrak S}_3,~~~
  \Lambda_{Q}/\Lambda_{X_3} \cong C_3,~~\mbox{ and } ~~
  G_{T}/\Lambda_{X_3} \cong C_2.
 \]
\begin{prop}
 \label{prop:geom}
 The four groups $G_{K_4}, \Lambda_Q, G_T$, and $\Lambda_{X_3}$
 acts on the honeycomb lattice graph ${\mathcal L}$ as
 graph automorphisms and we have
 \[
 K_4 \cong {\mathcal L}/G_{K_4},~~~
 Q\cong {\mathcal L}/\Lambda_{Q},~~~
 T \cong {\mathcal L}/G_{T},~~\mbox{ and } ~~
 X_3 \cong {\mathcal L}/\Lambda_{X_3}.
 \]
\end{prop}

\begin{proof}
We define a map $p_V:V({\mathcal L})\to {\mathfrak S}_4$ as follows.
Each edge $e$ carries a weight $\sigma_e\in \{(1, 4), (2, 4), (3, 4)\}\subset {\mathfrak S}_4$.
Since each edge $e$ is incident to a black vertex $b\in B$
and a white vertex $w\in W$ with $\|b-w\|=1$,
we define
\[
 \sigma_e = \begin{cases}
	(3, 4) & \mbox{ if } w = b + (-1,0),\\
	(1, 4) & \mbox{ if } w = b + (1/2,\sqrt{3}/2),\\
	(2, 4) & \mbox{ if } w = b + (1/2,-\sqrt{3}/2).
       \end{cases}
\]
Then define $p_V(-1,0) = {\rm id}\in {\mathfrak S}_4$.
For a vertex $v\in V({\mathcal L})$, we choose an
arbitral path $(e_1, e_2, \ldots, e_l)$ starting from $(1,0)$
and ending at $v$.
We define 
\[
 p_V(v) = \sigma_{e_1}\sigma_{e_2}\cdots \sigma_{e_l}\in {\mathfrak S}_4.
\]
It is easily confirmed that $p_V(v)$ is well defined,
that is, it does not depend on the choice of the path to $v$.
(See Figure $\ref{fig:truncatedtetrahedron}$.)
Further $p_V$ is periodic, that is,
$p_V(v)=p_V(w)$ if and only if
$
v-w \in \Lambda_{X_3}.
$
Thus we have $X_3\cong {\mathcal L}/\Lambda_{X_3}$.
 
The fundamental domain for the action of $G_{K_4}$ (resp. $\Lambda_{Q}$)
on ${\mathbb R}^2$ is shown in the left (resp. right) side of
Figure $\ref{fig:fundDomain}$, and
the fundamental domain for the action of $G_{T}$
on ${\mathbb R}^2$ is shown in the Figure $\ref{fig:fundDomainT}$.
From these figures we can confirm the statement of this proposition.
\begin{center}
 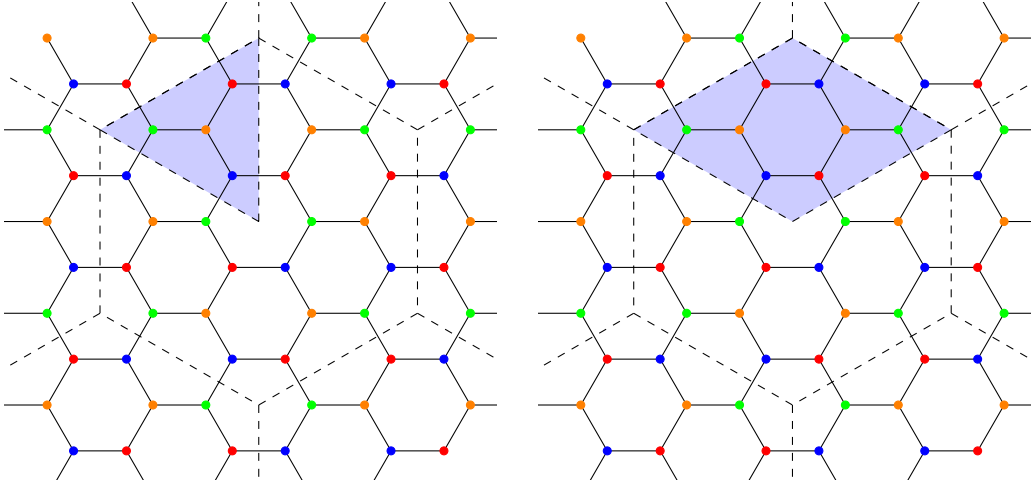
\begin{figure}[H]
  \begin{center}
   \begin{tikzpicture}[scale=0.7]
    \clip (-4.3,-4) -- (5,-4) -- (5,5) -- (-4.3,5) -- cycle;
    \draw[dashed,fill=blue, fill opacity=0.2] (0.5,{1.732/2}) -- +(0,{2*1.732}) -- +(150:{2*1.732}) -- cycle;
    \draw[dashed] (-2.5,{-sqrt(3)/2}) -- ++(-3,{-sqrt(3)});
    \draw[dashed] (-2.5,{-sqrt(3)/2})  ++(3,{-sqrt(3)}) -- ++(0,{-2*sqrt(3)});
    \draw[dashed] (-2.5,{-sqrt(3)/2})  ++(3,{-sqrt(3)})++(0,{4*sqrt(3)}) -- ++(0,{2*sqrt(3)});
    \draw[dashed] (-2.5,{-sqrt(3)/2})  ++(3,{-sqrt(3)}) ++(3,{sqrt(3)}) -- ++(3,{-sqrt(3)});
    \draw[dashed] (-2.5,{-sqrt(3)/2})  ++(6,{2*sqrt(3)}) -- ++(3,{sqrt(3)});
    \draw[dashed] (-2.5,{-sqrt(3)/2})  ++(0,{2*sqrt(3)}) -- ++(-3,{sqrt(3)});
    \foreach \y in {-3,...,3}{
    \foreach \x in {-3,...,4}{
    \draw ({(\x+\y)*3/2},{(\x-\y)*sqrt(3)/2}) -- ++(1,0);
    \draw ({(\x+\y)*3/2},{(\x-\y)*sqrt(3)/2}) -- ++({-1/2},{sqrt(3)/2});
    \draw ({(\x+\y)*3/2},{(\x-\y)*sqrt(3)/2}) -- ++({-1/2},{-sqrt(3)/2});
    }}
    \foreach \y in {-3,...,3}{
    \foreach \x in {-3,...,3}{
    \draw[red, fill=red] ({2*(\x+\y)*3/2},{2*(\x-\y)*sqrt(3)/2}) circle (0.075);
    \draw[blue, fill=blue] ({2*(\x+\y)*3/2},{2*(\x-\y)*sqrt(3)/2}) + (1,0) circle (0.075);
    \draw[green, fill=green] ({2*(\x+\y)*3/2},{2*(\x-\y)*sqrt(3)/2}) + ({-1/2},{sqrt(3)/2}) circle (0.075);
    \draw[orange, fill=orange] ({2*(\x+\y)*3/2},{2*(\x-\y)*sqrt(3)/2}) + ({-1/2},{-sqrt(3)/2}) circle (0.075);
    \draw[red, fill=red] ({2*(\x+\y)*3/2},{2*(\x-\y)*sqrt(3)/2}) + (-2,0)circle (0.075);
    \draw[blue, fill=blue] ({2*(\x+\y)*3/2},{2*(\x-\y)*sqrt(3)/2}) + (3,0) circle (0.075);
    \draw[green, fill=green] ({2*(\x+\y)*3/2},{2*(\x-\y)*sqrt(3)/2}) + ({3/2},{sqrt(3)/2}) circle (0.075);
    \draw[orange, fill=orange] ({2*(\x+\y)*3/2},{2*(\x-\y)*sqrt(3)/2}) + ({3/2},{-sqrt(3)/2}) circle (0.075);
    }}

    \draw[dashed] ({-2.5+3},{-sqrt(3)/2-sqrt(3)})
    -- ++(3,{sqrt(3)}) -- ++(0,{2*sqrt(3)}) -- ++(-3,{sqrt(3)}) -- ++(-3,{-sqrt(3)}) -- (-2.5,{-sqrt(3)/2}) -- ++(3,{-sqrt(3)});
   \end{tikzpicture}
   ~~
   \begin{tikzpicture}[scale=0.7]
    \clip (-4.3,-4) -- (5,-4) -- (5,5) -- (-4.3,5) -- cycle;
    \draw[dashed,fill=blue, fill opacity=0.2] (0.5,{1.732/2}) -- +(30:{2*1.732}) -- +(0,{2*1.732}) -- +(150:{2*1.732}) -- cycle;
    \draw[dashed] (-2.5,{-sqrt(3)/2}) -- ++(-3,{-sqrt(3)});
    \draw[dashed] (-2.5,{-sqrt(3)/2})  ++(3,{-sqrt(3)}) -- ++(0,{-2*sqrt(3)});
    \draw[dashed] (-2.5,{-sqrt(3)/2})  ++(3,{-sqrt(3)})++(0,{4*sqrt(3)}) -- ++(0,{2*sqrt(3)});
    \draw[dashed] (-2.5,{-sqrt(3)/2})  ++(3,{-sqrt(3)}) ++(3,{sqrt(3)}) -- ++(3,{-sqrt(3)});
    \draw[dashed] (-2.5,{-sqrt(3)/2})  ++(6,{2*sqrt(3)}) -- ++(3,{sqrt(3)});
    \draw[dashed] (-2.5,{-sqrt(3)/2})  ++(0,{2*sqrt(3)}) -- ++(-3,{sqrt(3)});
    \foreach \y in {-3,...,3}{
    \foreach \x in {-3,...,4}{
    \draw ({(\x+\y)*3/2},{(\x-\y)*sqrt(3)/2}) -- ++(1,0);
    \draw ({(\x+\y)*3/2},{(\x-\y)*sqrt(3)/2}) -- ++({-1/2},{sqrt(3)/2});
    \draw ({(\x+\y)*3/2},{(\x-\y)*sqrt(3)/2}) -- ++({-1/2},{-sqrt(3)/2});
    }}
    \foreach \y in {-3,...,3}{
    \foreach \x in {-3,...,3}{
    \draw[red, fill=red] ({2*(\x+\y)*3/2},{2*(\x-\y)*sqrt(3)/2}) circle (0.075);
    \draw[blue, fill=blue] ({2*(\x+\y)*3/2},{2*(\x-\y)*sqrt(3)/2}) + (1,0) circle (0.075);
    \draw[green, fill=green] ({2*(\x+\y)*3/2},{2*(\x-\y)*sqrt(3)/2}) + ({-1/2},{sqrt(3)/2}) circle (0.075);
    \draw[orange, fill=orange] ({2*(\x+\y)*3/2},{2*(\x-\y)*sqrt(3)/2}) + ({-1/2},{-sqrt(3)/2}) circle (0.075);
    \draw[red, fill=red] ({2*(\x+\y)*3/2},{2*(\x-\y)*sqrt(3)/2}) + (-2,0)circle (0.075);
    \draw[blue, fill=blue] ({2*(\x+\y)*3/2},{2*(\x-\y)*sqrt(3)/2}) + (3,0) circle (0.075);
    \draw[green, fill=green] ({2*(\x+\y)*3/2},{2*(\x-\y)*sqrt(3)/2}) + ({3/2},{sqrt(3)/2}) circle (0.075);
    \draw[orange, fill=orange] ({2*(\x+\y)*3/2},{2*(\x-\y)*sqrt(3)/2}) + ({3/2},{-sqrt(3)/2}) circle (0.075);
    }}

    \draw[dashed] (-2.5,{-sqrt(3)/2}) -- ++(3,{-sqrt(3)})
    -- ++(3,{sqrt(3)}) -- ++(0,{2*sqrt(3)}) -- ++(-3,{sqrt(3)}) -- ++(-3,{-sqrt(3)}) -- cycle;
   \end{tikzpicture}
   \caption{Fundamental domain for the action of $G_{K_4}$ (left)
   and $\Lambda_{Q}$ (right)}
   \label{fig:fundDomain}
  \end{center}
 \end{figure}
\end{center}

\begin{center}
 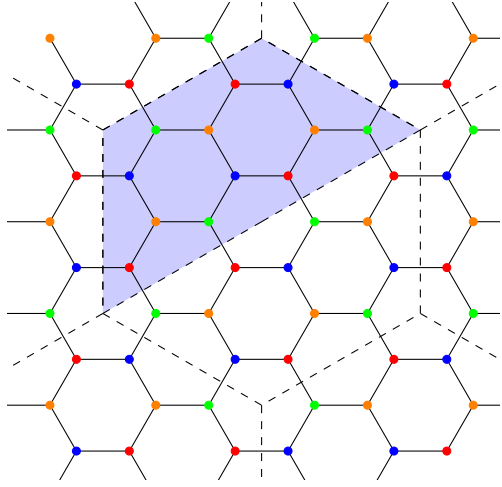
\begin{figure}[H]
  \begin{center}
   \begin{tikzpicture}[scale=0.7]
    \clip (-4.3,-4) -- (5,-4) -- (5,5) -- (-4.3,5) -- cycle;
    \draw[dashed,fill=blue, fill opacity=0.2] (0.5,{1.732/2}) -- ++(30:{2*sqrt(3)}) -- ++(150:{2*sqrt(3)}) -- ++(210:{2*sqrt(3)}) -- ++(270:{2*sqrt(3)})-- cycle;
    \draw[dashed] (-2.5,{-sqrt(3)/2}) -- ++(-3,{-sqrt(3)});
    \draw[dashed] (-2.5,{-sqrt(3)/2})  ++(3,{-sqrt(3)}) -- ++(0,{-2*sqrt(3)});
    \draw[dashed] (-2.5,{-sqrt(3)/2})  ++(3,{-sqrt(3)})++(0,{4*sqrt(3)}) -- ++(0,{2*sqrt(3)});
    \draw[dashed] (-2.5,{-sqrt(3)/2})  ++(3,{-sqrt(3)}) ++(3,{sqrt(3)}) -- ++(3,{-sqrt(3)});
    \draw[dashed] (-2.5,{-sqrt(3)/2})  ++(6,{2*sqrt(3)}) -- ++(3,{sqrt(3)});
    \draw[dashed] (-2.5,{-sqrt(3)/2})  ++(0,{2*sqrt(3)}) -- ++(-3,{sqrt(3)});
    \foreach \y in {-3,...,3}{
    \foreach \x in {-3,...,4}{
    \draw ({(\x+\y)*3/2},{(\x-\y)*sqrt(3)/2}) -- ++(1,0);
    \draw ({(\x+\y)*3/2},{(\x-\y)*sqrt(3)/2}) -- ++({-1/2},{sqrt(3)/2});
    \draw ({(\x+\y)*3/2},{(\x-\y)*sqrt(3)/2}) -- ++({-1/2},{-sqrt(3)/2});
    }}
    \foreach \y in {-3,...,3}{
    \foreach \x in {-3,...,3}{
    \draw[red, fill=red] ({2*(\x+\y)*3/2},{2*(\x-\y)*sqrt(3)/2}) circle (0.075);
    \draw[blue, fill=blue] ({2*(\x+\y)*3/2},{2*(\x-\y)*sqrt(3)/2}) + (1,0) circle (0.075);
    \draw[green, fill=green] ({2*(\x+\y)*3/2},{2*(\x-\y)*sqrt(3)/2}) + ({-1/2},{sqrt(3)/2}) circle (0.075);
    \draw[orange, fill=orange] ({2*(\x+\y)*3/2},{2*(\x-\y)*sqrt(3)/2}) + ({-1/2},{-sqrt(3)/2}) circle (0.075);
    \draw[red, fill=red] ({2*(\x+\y)*3/2},{2*(\x-\y)*sqrt(3)/2}) + (-2,0)circle (0.075);
    \draw[blue, fill=blue] ({2*(\x+\y)*3/2},{2*(\x-\y)*sqrt(3)/2}) + (3,0) circle (0.075);
    \draw[green, fill=green] ({2*(\x+\y)*3/2},{2*(\x-\y)*sqrt(3)/2}) + ({3/2},{sqrt(3)/2}) circle (0.075);
    \draw[orange, fill=orange] ({2*(\x+\y)*3/2},{2*(\x-\y)*sqrt(3)/2}) + ({3/2},{-sqrt(3)/2}) circle (0.075);
    }}

    \draw[dashed] (-2.5,{-sqrt(3)/2}) -- ++(3,{-sqrt(3)})
    -- ++(3,{sqrt(3)}) -- ++(0,{2*sqrt(3)}) -- ++(-3,{sqrt(3)}) -- ++(-3,{-sqrt(3)}) -- cycle;
   \end{tikzpicture}
   \caption{Fundamental domain for the action of $G_T$}
   \label{fig:fundDomainT}
  \end{center}
 \end{figure}
\end{center}

\end{proof}

To our knowledge, 
the following property of the honeycomb lattice 
has not been explicitly stated in the literature, 
and it is compared with the results in \cite{kotani2001standard,sunada2012topological}.
\begin{cor}
 The honeycomb lattice graph ${\mathcal L}$ is a ${\mathbb Z}^2$-cover of
 the star graph $X_3$ and the cube $Q$, and is also
 a ${\mathbb Z}^2\rtimes C_2$-cover of the tetrahedron $K_4$ and the truncated tetrahedron $T$.
\end{cor}

\begin{cor}
 The spectrum ${\mathcal S}$ of $X_3$ is given by
\[
 {\mathcal S}=\left\{\pm\left|1+\exp\left({\frac{2\pi{k}{i}}{6}}\right) +
 \exp\left({\frac{2\pi{(-k+3l)}{i}}{6}}\right) \right|~~\middle|\,
 k\in {\mathbb Z}/6{\mathbb Z},~ l\in {\mathbb Z}/2{\mathbb Z}\right\}.
\]
\end{cor}
\begin{proof}
 Let the hexagonal lattice graph 
 ${\mathcal L} = (V({\mathcal L}), E({\mathcal L}))$,
 $B, W, {\mathbf v}_1$ and ${\mathbf v}_2$
 be defined as in $(\ref{eq:BW})$ and $(\ref{eq:defbase})$.
 We express each vertex $v$ of ${\mathcal L}$ as
 a member in  ${\mathbb Z}^2\times\{0,1\}$ by using the bijections
 \[
 B \ni (1,0) + x{\mathbf v}_1 + y{\mathbf v}_2 \mapsto (x,y,0) \in {\mathbb Z}^2\times\{0\},
 \]
 and
 \[
 W \ni (0,0) + x{\mathbf v}_1 + y{\mathbf v}_2 \mapsto (x,y,1) \in {\mathbb Z}^2\times\{1\}.
 \]
 
 We define  the functions $f_{k,l}:V({\mathcal L}) \to {\mathbb C}$ for
 $k\in {\mathbb Z}/6{\mathbb Z}, l\in {\mathbb Z}/2{\mathbb Z}$ by
 \[
  f_{k,l}(x,y,t) = \begin{cases}
		    \alpha (-1)^{ly}\zeta_6^{k(x-y)} & t = 0,\\
		    \beta (-1)^{ly}\zeta_6^{k(x-y)} & t = 1,
		   \end{cases}
 \]
 where $\alpha, \beta \in {\mathbb C}$ and $\zeta_6 = \exp\left(\frac{2\pi i}{6}\right)$.
 As shown in Proposition $\ref{prop:geom}$, $X_3$ is isomorphic to the quotient graph
 ${\mathcal L}/\Lambda_{X_3}$ where $\Lambda_{X_3} = {\mathbb Z}(2{\mathbf v}_1 + 2{\mathbf v}_2) + {\mathbb Z}(4{\mathbf v}_1-2{\mathbf v}_2)$.
 Therefore $f_{k,l}$ can be considered as a function defined on the
 quotient graph, since
 $f_{k,l}(v+2{\mathbf v}_1 + 2{\mathbf v}_2)=f_{k,l}(v+4{\mathbf v}_1-2{\mathbf v}_2)=f_{k,l}(v)$
 for all $v\in V({\mathcal L})$.
 For $f_{k,l}$ to be an eigenfunction of the adjacency operator
 of $X_3$, the vector $(\alpha,\beta)$ must satisfy 
 \[
 \begin{pmatrix}
  0 & 1 + \zeta_6^k + (-1)^l\zeta_6^{-k} \\
  1 + \zeta_6^{-k} + (-1)^l\zeta_6^{k} & 0
 \end{pmatrix}
 \begin{pmatrix}
  \alpha \\ \beta
 \end{pmatrix}
 =
 \lambda
 \begin{pmatrix}
  \alpha \\ \beta
 \end{pmatrix},
 \]
 from which we obtain the
 eigenvalues $\lambda = \pm \left|1 + \zeta_6^k + (-1)^l\zeta_6^{-k}\right|$.
\end{proof}

\section{Concluding comments}
In the case of the ${\mathfrak S}_3$-cover, 
the relation among Ihara zeta functions of intermediate covers already 
determines non-trivial spectral relations. 
For larger groups such as ${\mathfrak S}_4$ or ${\mathfrak S}_5$,
there exist many intermediate covers, 
and it is natural to ask whether the spectra of 
them can be reconstructed systematically 
from the zeta relations and Theorem $\ref{thm:chapuyferay}$. 
We leave this direction for future research.


\begin{thebibliography}{10}

\bibitem{abdollahi2009cayley}
Alireza Abdollahi and Ebrahim Vatandoost.
\newblock Which cayley graphs are integral?
\newblock {\em {T}he {E}lectronic {J}ournal of {C}ombinatorics}, pages
  R122--R122, 2009.

\bibitem{bass1992ihara}
Hyman Bass.
\newblock The {I}hara-{S}elberg zeta function of a tree lattice.
\newblock {\em International Journal of Mathematics}, 3(06):717--797, 1992.

\bibitem{ChapuyFeray}
Guillaume Chapuy and Valentin F{\'e}ray.
\newblock A note on a cayley graph of {${\mathfrak S}_n$}.
\newblock {\em arXiv preprint arXiv:1202.4976}, 2012.

\bibitem{ehrlich1973loopless}
Gideon Ehrlich.
\newblock Loopless algorithms for generating permutations, combinations, and
  other combinatorial configurations.
\newblock {\em Journal of the ACM (JACM)}, 20(3):500--513, 1973.

\bibitem{ihara1966discrete}
Yasutaka Ihara.
\newblock On discrete subgroups of the two by two projective linear group over
  p-adic fields.
\newblock {\em Journal of the Mathematical Society of Japan}, 18(3):219--235,
  1966.

\bibitem{ishikawa2024non}
Masao Ishikawa, Fumihiko Nakano, and Taizo Sadahiro.
\newblock Non-isomorphic cayley graphs with same random walk distributions.
\newblock {\em arXiv preprint arXiv:2408.01666}, 2024.

\bibitem{knuth2011taocp4a}
Donald~E. Knuth.
\newblock {\em The Art of Computer Programming, Volume 4A: Combinatorial
  Algorithms, Part 1}.
\newblock Addison-Wesley, Boston, 2011.

\bibitem{kotani2001standard}
Motoko Kotani and Toshikazu Sunada.
\newblock Standard realizations of crystal lattices via harmonic maps.
\newblock {\em Transactions of the American Mathematical Society},
  353(1):1--20, 2001.

\bibitem{krakovski2012spectrum}
Roi Krakovski and Bojan Mohar.
\newblock Spectrum of cayley graphs on the symmetric group generated by
  transpositions.
\newblock {\em Linear Algebra and its applications}, 437(3):1033--1039, 2012.

\bibitem{stark2000zeta}
Harold~M Stark and Audrey~A Terras.
\newblock Zeta functions of finite graphs and coverings, part {II}.
\newblock {\em Advances in Mathematics}, 154(1):132--195, 2000.

\bibitem{sunada2012topological}
Toshikazu Sunada.
\newblock {\em Topological crystallography: with a view towards discrete
  geometric analysis}, volume~6.
\newblock Springer Science \& Business Media, 2012.

\bibitem{terras2010zeta}
Audrey Terras.
\newblock {\em Zeta functions of graphs: a stroll through the garden}, volume
  128.
\newblock Cambridge University Press, 2010.

\end{thebibliography}

\end{document}